\documentclass[final,a4paper,leqno]{svjour3}

\usepackage[english]{babel}
\usepackage{amsmath,amssymb,amsfonts,graphicx,subfigure,float}
\usepackage[ruled]{algorithm2e} 
\usepackage{myshortcuts}

\usepackage{pgfplots}
\usetikzlibrary{external}
\usetikzlibrary{decorations.pathreplacing}

\usepackage{pifont}
\definecolor{myblue}{RGB}{116,173,209}
\definecolor{myred}{RGB}{244,109,67}

\newcommand{\ejcomment}[1]{{\color{black}#1}}
\newenvironment{changed}{\color{black}}{\color{black}}
\newcommand{\vpcomment}[1]{{\color{black}#1}}

\iftrue
  \usepackage{tikz}
  \usepackage{pgfplots}
  \pgfplotsset{
    compat=newest,
    tick label style={font=\scriptsize},
    label style={font=\scriptsize},
    legend style={font=\scriptsize}
  }

  \pgfplotsset{select coords between index/.style 2 args={
      x filter/.code={
          \ifnum\coordindex<#1\fi
          \ifnum\coordindex>#2\fi
      }
  }}
  \usetikzlibrary{decorations.pathreplacing}
  
  \iftrue 
     \usepgfplotslibrary{external} 
  \else
     \renewcommand{\tikzsetnextfilename}[1]{}
  \fi
\else
  \usepackage{tikzexternal}
  \tikzsetexternalprefix{gfxtmp/}
  
\fi

\title{From eigenvector nonlinearities with quadratic structure to eigenvalue nonlinearities with algebraic structure  
}
\author{Elias Jarlebring \and Vilhelm P. Lithell}
\institute{Elias Jarlebring (corresponding author)\at
              Department of Mathematics, KTH Royal Institute of Technology, Stockholm, Sweden \\
              \email{eliasj@kth.se}           
           \and
           Vilhelm P. Lithell \at Department of Mathematics, KTH Royal Institute of Technology, Stockholm, Sweden 
}

\date{Received: date / Accepted: date }

\begin{document}

\maketitle
\begin{abstract}
Over the past decades, transformations between different classes of eigenvalue problems have played a central role in the development of numerical methods for eigenvalue computations. One of the most well-known and successful examples of this is the companion linearization for polynomial eigenvalue problems. In this paper, we construct a transformation that equivalently reframes a specific type of eigenvalue problem with eigenvector nonlinearities (NEPv) into 
an eigenvalue problem with eigenvalue nonlinearities (NEP).
The NEPv class considered consists of nonlinearities expressed as sums of products of matrices and scalar functions, where the scalar functions depend nonlinearly on the eigenvector.
Our transformation defines scalar eigenvalue 
nonlinearities through a polynomial system, resulting in NEP nonlinearities of algebraic type. 
We propose methods to
solve the polynomial system, one of which involves a multiparameter eigenvalue problem (MEP). 
We adapt well-established NEP solvers to this setting, with the most effective strategy being a combination of deflation and a locally quadratically convergent iterative method. 
The efficiency and properties of the approach are illustrated by solving a problem related to a modification of a Gross-Pitaevskii equation (GPE).
The simulations are reproducible and publicly available.
\end{abstract}

\section{Introduction}\label{sec:intro}
\begin{changed}
    The general context of this paper concerns a nonlinear eigenvalue problem, where the nonlinearity appears as eigenvector dependence. 
More precisely, we want to find 
eigenpairs $(\lambda, v)\in\RR\times\RR^{n}\backslash\{0\}$, such that
\begin{subequations}\label{eq:full_NEPv_problem}
\begin{eqnarray}
    \lambda E v&=& A(v)v,\label{eq:NEPv}\\
    1&=& \norm{v}_B^2 := \trans{v}Bv, 
    \label{eq:NEPv_norm_cond}
\end{eqnarray}
\end{subequations}
where $A(v)\in\RR^{n\times n}$ is symmetric for all $v$ and $E,B\in\RR^{n\times n}$ are symmetric positive definite matrices. In most situations, the problem can be equivalently transformed to a problem of the same structure with $E=B=I$; however, we prefer the general structure for convenience in the application. 

This is a very general problem, and there are numerous applications; see the literature discussion below. 
One of the most important
applications stems from the modeling of Bose-Einstein condensates
by the Gross-Pitaevskii equations (GPE) \cite{Henning:2025:REVIEW}.  Inspired by the structure arising in the discretization of the GPE, we consider a problem where the nonlinearities are 
quadratic functions of the eigenvectors. More precisely, the NEPv \eqref{eq:NEPv}  is defined by
\begin{equation}
    A(v):=A_0+(\trans{a_1}v)^2 a_1\trans{a_1}+\cdots+(\trans{a_m}v)^2 a_m\trans{a_m},\label{eq:NEPv}
\end{equation}
where $A_0\in\RR^{n\times n}$ is a symmetric matrix, and $a_1,\dots,a_m\in\RR^{n}$. To avoid degeneracy, we assume that the pencil $\lambda E-A_0$ is regular, i.e., $\det(\lambda E-A_0)\not\equiv 0$, which implies that $\lambda E-A_0$ is invertible except for a finite number of $\lambda$-values. In the context of discretization of PDEs, $E$ is usually symmetric positive definite, which implies that the pencil is regular.

One way to see how this structure arises from the GPE is by noting that the nonlinearity appearing the GPE is quadratic in the eigenfunction,
due to the energy associated with particle interactions \cite[Section~2.2]{Henning:2025:REVIEW}.
The finite difference discretization of the nonlinearity involves the term $\diag(v)^2=(e_1^Tv)^2e_1e_1^T+\cdots+(e_n^Tv)^2e_ne_n^T$; cf 
\cite[Eq (5.4)]{Jarlebring:2014:INVIT}, i.e., the same structure that appears in \eqref{eq:NEPv}.
We consider the  slightly more general case $a_i\neq e_i$, since it directly arises from continuous formulation of a similar problem; see Section~\ref{sec:num_examples}, and it also enables us to naturally introduce scaling of the nonlinearity. Our theoretical results hold for any choice of $m$, i.e., also $m=n$ which is the most natural choice in the GPE application. However, in practice, the performance of the proposed numerical methods is much better when $m\ll n$, and infeasible for a fine finite-difference discretization with $m=n$. 
Despite this slight mismatch between the GPE application and the 
proposed method, we view the results of this paper as directly valuable for the problem \eqref{eq:NEPv}, and as a stepping stone for further method development for the GPE.

The main line of reasoning and scientific contributions is as follows.
\begin{itemize}
    \item[(a)] We propose a technique that equivalently transforms \eqref{eq:full_NEPv_problem} with \eqref{eq:NEPv}, to a problem with  eigenvalue nonlinearities. More precisely, we derive the equivalent equation 
\begin{equation}\label{eq:NEP}
 (A_0-\lambda E + \mu_1^2(\lambda) a_1\trans{a_1}+\cdots+\mu_m^2(\lambda) a_m\trans{a_m})v = 0, 
\end{equation}
where $\mu_1(\lambda), \dots, \mu_m(\lambda)$ are scalar functions of the eigenvalue. These scalar functions are determined from the solution to a parameterized polynomial system, and are therefore algebraic functions of $\lambda$.
\item[(b)] We observe that \eqref{eq:NEP} belongs to the much more common class of nonlinear eigenvalue problems (NEPs), where the nonlinearity  appears as a dependence on the eigenvalue.  This allows  us to leverage and specialize the use of efficient methods for these problems; see, e.g.,  \cite{Guttel2017,Mehrmann:2004:NLEVP,Ruhe:1973:NLEVP} for methods for NEPs. 
\end{itemize}
Decades of research in numerical linear algebra have produced robust
NEP-solvers capable of computing many eigenpairs for various types of large problems in a reliable way.
Therefore, one of the most important features of this combination of (a) and (b) lies in the fact that it allows us to compute several eigenpairs, in a reliable way. This  is very rare among methods for the NEPv, as well as for the GPE, where essentially it can be interpreted as the computation of excited states.

The formulation of \eqref{eq:NEP} involves a system of polynomial equations. More precisely, the functions $\mu_1,\ldots,\mu_m$ are implicitly defined from a parameterized polynomial system (derived in Section~\ref{sec:equivalence}), whose degree and size grows with $m$, but is independent of $n$. Hence, for every evaluation of the left-hand side of \eqref{eq:NEP} we must solve this system. The reliable computation of the solution requires particular attention. We therefore characterize the system and propose methods to solve it analytically for small $m$, and for larger $m$ (but still $m\ll n$) we explain how techniques from computational algebraic geometry can be used, e.g.,   \cite{HomotopyContinuation.jl}. We also specialize methods from multiparameter eigenvalue problem techniques, inspired by methods implemented in \cite{multipareig}. Theory and methods concerning this are provided in Section~\ref{sec:solving_poly_sys}.


The related literature can be classified as GPE-specific and 
method development driven from a numerical linear algebra perspective.
The the GPE, gradient-flow-based methods are among the most widely used \cite{Bao:2004:BOSEEINSTEIN} and, to the best of our knowledge,
constitute the state of the art.  The gradient flow approach exists in various forms, 
and in the context of a spatial finite element discretization, the choice of Sobolev space plays a crucial role \cite{Henning:2020:SOBOLEV}. See further literature discussion of gradient flow methods in \cite[Section~5]{Henning:2025:REVIEW}. We note that most results concerning the gradient flow, depend on results from the PDE discretization, usually a finite-element discretization, and in contrast, we assume a given discretization (of a GPE-like problem) and develop methods from a numerical linear algebra perspective.


An extension of the inverse iteration algorithm is discussed in \cite{Jarlebring:2014:INVIT}, with its convergence thoroughly examined in \cite{Hen22}. This algorithm does have an interpretation as a variation of a gradient flow 
\cite[Section~4]{Jarlebring:2014:INVIT}, and it can also be interpreted as an implicit iterative method \cite{Jarlebring:2021:IMPLICIT}.

There are works that target specific structures for the NEPv in \eqref{eq:full_NEPv_problem} that feature linearizations, in the sense of a companion linearization, as explored in \cite{Claes:2022:NEPvlin}. The work presented in \cite{Claes:2023:CONTOUR} engages with polynomial nonlinearities similar to our context, merging techniques from computational algebraic geometry with contour integration.  We note that, similar to our results, these methods can find several eigenpairs, but the assumptions on the structure of the nonlinearity are substantially different from our GPE-like structure. The NEPv has also been studied from a more theoretical perspective, particularly within the convergence theory of the self-consistent field (SCF) iteration, as documented in \cite{Yang:2009:SCF,Upadhyaya:2021:DENSITY,Bai:2022:SHARP,Lu:2024:LOCALLY}, and \cite{Bai:2024:VARIATIONAL}. Additionally, perturbation theory results have been reported in \cite{Cai:2020} and \cite{Truhar:2021}.

For completeness we now also describe some situations where the NEPv \eqref{eq:full_NEPv_problem} arises, but not necessarily with the structure \eqref{eq:NEPv}.
The problem appears in data science applications such as \cite{hein},\cite{Hein:2009:PLAPLACIAN}, where it stems from a variation of spectral clustering, 
resulting from a  regularization of the (standard) 2-Laplacian based clustering. One of the most competitive approaches for that problem class is based on Grassmannian optimization \cite{Pasadakis:2022:GRASSMANN}.  See also further data science applications in \cite{tudisco2018}. 
The problem also naturally arises in various theoretical and mathematical applications, notably in certain optimization problems. For instance, the Rayleigh quotient minimization techniques discussed in \cite{Lu2020}, and \cite{Bai2018} addresses NEPvs of the form \eqref{eq:full_NEPv_problem} by iteratively solving linear eigenvalue problems. Additionally, the backward error of specific eigenvalue nonlinearities - specifically, rational nonlinear eigenvalue problems using Rosenbrock linearization - can be characterized through NEPvs, as demonstrated in \cite{Lu:2024:ROSENBROCK}.

The paper is organized as follows. Section~\ref{sec:transformation} discusses the transformation of the problem into a system of the form \eqref{eq:NEP}, where the functions $\mu_1(\lambda),\ldots,\mu_m(\lambda)$ are defined as solutions to a $\lambda$-dependent polynomial system. Efficient and reliable solutions to this polynomial system are crucial in order to apply NEP methods. In Section~\ref{sec:solving_poly_sys}, we address numerical techniques for solving the system efficiently. For the case of $m=2$, we derive an analytical solution, while for the general case, we present an approach based on multiparameter eigenvalue problems (MEPs).  Section~\ref{sec:implementation} provides implementation details on how to effectively combine MEP and NEP methods to enable the computation of multiple eigenvalues through deflation. Drawing on established properties of NEP methods, we propose a specific NEP method for our context. Finally, numerical simulations of a GPE-like problem are presented in Section~\ref{sec:num_examples}.

\section{Transformation}\label{sec:transformation}
\subsection{A small introductory example}
To illustrate the technique and the derivation, we consider the following small example containing only one nonlinear term
\begin{equation}\label{eq:first_example}
    \lambda v=\left(A_0+(\trans{v}a_1)^2a_1\trans{a_1}\right)v,\qquad \|v\|=1,
\end{equation}
with $A_0\in\RR^{2\times 2}$ and $a_1\in\RR^{2}$ given by
\begin{equation}\label{eq:first_example_matrices}
    A_0:=
\begin{bmatrix}
4 & 1 \\
1 & 6
\end{bmatrix},\quad
a_1:=\begin{bmatrix}
  3\\[2pt]2
\end{bmatrix}.
\end{equation}
Let $\mu:=\trans{v}a_1$. Rearranging \eqref{eq:first_example} gives
\[
(\lambda I-A_0)v=\mu^3 a_1,\qquad\text{so}\qquad
v=\mu^3(\lambda I-A_0)^{-1}a_1,
\]
whenever $\lambda I-A_0$ is invertible. 
Using $\|v\|=1$ yields
\[
1=\|v\|^2=\mu^6\,\trans{a_1}(\lambda I-A_0)^{-2}a_1,
\]
which can be solved for $\mu^2$ as a function of $\lambda$:
\begin{equation}\label{eq:mu1_simple_example}
f(\lambda):=\mu^2=
\left(\frac{(\lambda^2 - 10\lambda + 23)^2}{13\lambda^2 - 116\lambda + 281}\right)^{1/3}.
\end{equation}
The function $f(\lambda)$ is visualized in Figure~\ref{fig:mu-sing}. Substituting $\mu^2=f(\lambda)$ back into \eqref{eq:first_example} yields the NEP
\begin{equation}\label{eq:first_example_NEP}
    0=(A_0-\lambda I+f(\lambda) \trans{a_1a_1})v.
\end{equation}
By construction, any solution of \eqref{eq:first_example} solves \eqref{eq:first_example_NEP}; conversely, as shown below, any solution of \eqref{eq:first_example_NEP} that is not an eigenvalue of $A_0$, also solves \eqref{eq:first_example}.

In principle, almost any NEP-solver could be used to solve \eqref{eq:first_example_NEP}. Using a Newton method with deflation for the NEP (see Section~\ref{sec:deflation}), we compute the two eigenvalues $\lambda\approx 4.2175$ and $\lambda\approx 174.5385$, which constitute all the solutions to \eqref{eq:first_example}. 
A key message of this paper is that NEPvs are typically harder to solve reliably than NEPs. In the present example, without prior information about the large eigenvalue, it could easily be missed, e.g., with the repeated application of a locally convergent NEPv algorithm one could easily miss it due to reconvergence to a previously found solution. 


\begin{figure}[H]
    \centering
    \includegraphics{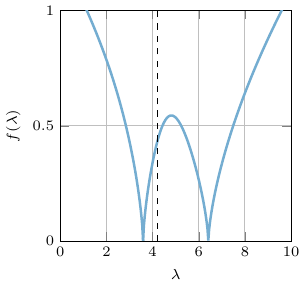}
    \caption{\ejcomment{The function $f(\lambda)$ in \eqref{eq:mu1_simple_example} and an eigenvalue, indicated by the dashed line.}}
    \label{fig:mu-sing}
\end{figure}
\end{changed}
\subsection{A parametrized polynomial system}\label{sec:param_pol_sys}
We derive the equations defining the NEP \eqref{eq:NEP},
by considering a particular transformation 
of \eqref{eq:full_NEPv_problem} that eliminates the presence of $v$, and instead, the nonlinearity appears only in terms of the scalars
\begin{equation}    
\mu_i=\trans{a_i}v,\;\;\;i=1,\ldots,m,\label{eq:mui_def}
\end{equation}
which are later interpreted as functions of $\lambda$.
By multiplying \eqref{eq:NEPv} from the left with $(\lambda E -A_0)^{-1}$ we obtain a formula for $v$
consisting of  a linear combination of $\lambda$-dependent \ejcomment{vectors
\begin{equation}\label{eq:v_factorized}
  v=\mu_1^3 (\lambda E -A_0)^{-1}a_1+\cdots +\mu_m^3(\lambda E -A_0)^{-1}a_m. 
\end{equation}
Note that the right-hand side of \eqref{eq:v_factorized} is independent of $v$.
By applying the normalization condition \eqref{eq:NEPv_norm_cond} on the vector $v$, we obtain
\begin{equation}
\begin{split}\label{eq:normalization_relation}
    1&=\trans{v}Bv\\
    &=\sum_{i,j\leq m} \mu_i^3\mu_j^3\trans{a_i}(\lambda E -A_0)^{-1}B(\lambda E -A_0)^{-1}a_j .
\end{split}
\end{equation}
}
Hence, the normalization gives a scalar condition involving the $m$ scalar functions $\mu_1,\ldots,\mu_m$, as well as the scalar $\lambda$.

We can obtain further polynomial relations between the scalars by multiplying equation \eqref{eq:v_factorized} from the left by $\trans{a_\ell}$, 
\ejcomment{
leading to 
\begin{equation}\label{eq:mui_relation}
    \mu_\ell=\trans{a_\ell}v=
    \mu_1^3\trans{a_\ell}(\lambda E -A_0)^{-1}a_1+
    \cdots+
    \mu_m^3\trans{a_\ell}(\lambda E -A_0)^{-1}a_m,
\end{equation}}
for any $\ell=1,\ldots,m$.
To state these relations more concisely, let $\mu\in\RR^m$ be the vector
\begin{equation}
    \mu = \trans{[\mu_1, \dots, \mu_m]},
\end{equation}
and let $\mu^p$, denote element-wise exponentiation of $\mu$.
Let $A_m$ be the matrix
\begin{equation}
    A_m := \left[a_1, \dots, a_m\right]\in\RR^{n\times m}.
\end{equation}
Using this new notation, the relations \eqref{eq:normalization_relation} and \eqref{eq:mui_relation} are summarized in a more concise way in the following proposition.

\begin{proposition}\label{prop:polysys_satisfied}
    Let $G(\lambda)$ be a matrix whose elements depend on the parameter $\lambda$, defined by
    \begin{equation}\label{eq:G_def}
        G(\lambda) := \trans{A_m}(\lambda E-A_0)^{-1}B(\lambda E-A_0)^{-1}A_m,
    \end{equation}
    and similarly let $H(\lambda)$ be 
    \begin{equation}\label{eq:H_def}
        H(\lambda) := \trans{A_m}(\lambda E-A_0)^{-1}A_m.
    \end{equation} If $(\lambda, v)$ is an eigenpair of \eqref{eq:full_NEPv_problem}, and $\lambda$ is not an eigenvalue of the pencil $A_0-\lambda E$, then the following equations are satisfied
    \begin{subequations}\label{eq:poly_sys_mat_form}
    \begin{eqnarray}
        & &\trans{(\mu^3)}G(\lambda)\mu^3 - 1 = 0\label{eq:norm_equations_mat_form},\\
        & &H(\lambda)\mu^3 - \mu = 0. \label{eq:mui_relations_mat_form}
    \end{eqnarray}
    \end{subequations}
\end{proposition}
\ejcomment{\begin{proof} We note that $\lambda E-A_0$ is invertible if $\lambda$ is not an eigenvalue of the pencil, such that \eqref{eq:normalization_relation} and \eqref{eq:mui_relation} can be evaluated. The direct collection of the terms in 
\eqref{eq:normalization_relation} and \eqref{eq:mui_relation} yields 
\eqref{eq:norm_equations_mat_form} and \eqref{eq:mui_relations_mat_form}. 
\end{proof}}
\vpcomment{The system \eqref{eq:poly_sys_mat_form} has $m+1$ unknowns and $m+1$ equations.} 
\ejcomment{The main scientific contribution of this paper is a transformation of the NEPv \eqref{eq:full_NEPv_problem} with \eqref{eq:NEPv} to a NEP \eqref{eq:NEP}, and in order to achieve this, we need to construct functions $\mu_1(\lambda),\ldots,\mu_m(\lambda)$.  For the purpose of obtaining such a  transformation we now neglect one of the equations, such that we obtain $m$ equations in $m+1$ unknowns.
Neglecting the normalization equation \eqref{eq:norm_equations_mat_form}, is unsuitable, and one can show that this yields a singular system. 
The remaining equations \eqref{eq:mui_relations_mat_form} can be reordered arbitrarily. We can therefore, without loss of generality, assume that we remove the last equation in \eqref{eq:mui_relations_mat_form}, which can formally be represented as the multiplication with
\[
P=\begin{bmatrix}I & 0\end{bmatrix}\in\RR^{(m-1)\times m},
\]
where $I\in\RR^{(m-1)\times (m-1)}$ is the identity matrix.  
This yields the equations}
\begin{subequations}\label{eq:reduced_syst_mat_form}
\begin{eqnarray}
    & &\trans{(\mu^3)}G(\lambda)\mu^3-1 = 0,\\
    & &P\left(H(\lambda)\mu^3-\mu\right) = 0\label{eq:P_eqns}.
\end{eqnarray}
\end{subequations}
In the next section we will use these equations to define the functions $\mu_1(\lambda),\dots,\mu_m(\lambda)$. Conditions on the existence of implicit functions can be determined from the Jacobian of the system of equations, which in this case is explicitly given by the following formula.
\begin{proposition}\label{thm:jacobian}
    The Jacobian of the reduced system \eqref{eq:reduced_syst_mat_form} with respect to $\mu$ at a point $(\lambda, \mu)$, where $\lambda$ is not an eigenvalue of the pencil $A_0-\lambda E$, is given by
    \begin{equation}\label{eq:system_jac}
        J(\lambda, \mu) = 
        \begin{bmatrix}
            6\trans{(\mu^3)}G(\lambda)\diag(\mu^2) \\ 
            P(3H(\lambda)\diag(\mu^2) - I)
        \end{bmatrix} \in \RR^{m\times m},
    \end{equation}
    where $\diag(\mu^2)$ denotes the diagonal matrix whose diagonal elements are those of $\mu^2$.
\end{proposition}
\begin{proof}
    Fix $\lambda$ to be as above.
    Since $G(\lambda)$ is symmetric, we have
    \begin{equation}
        \frac{\partial}{\partial \mu}\left[\trans{(\mu^3)}G(\lambda)\mu^3\right] = \trans{(\mu^3)}(G(\lambda) + \trans{G(\lambda)})3\diag(\mu^2) = 6\trans{(\mu^3)}G(\lambda)\diag(\mu^2),
    \end{equation}
    which gives the first row in \eqref{eq:system_jac}.
    For the remaining entries in \eqref{eq:system_jac}, consider the derivative with respect to $\mu_\ell$, $1\leq \ell \leq m$, of row $k$ of \eqref{eq:mui_relations_mat_form}
    \begin{equation}
        \frac{\partial}{\partial \mu_\ell}\left(h_{k1}\mu_1^3 + \dots + h_{km}\mu_m^3 - \mu_m \right) = 3h_{k\ell}\mu_\ell^2 - \delta_{\ell m}, \quad 1\leq k \leq m,
    \end{equation}
    with $\delta_{ij}$ denoting the Kronecker delta.
    Then the Jacobian of \eqref{eq:mui_relations_mat_form} with respect to $\mu$ becomes 
    \begin{equation}
        \frac{\partial}{\partial \mu}\left[H(\lambda)\mu^3-\mu\right] = 3H(\lambda)\diag(\mu^2) - I,
    \end{equation}
    and selecting $m-1$ rows from this Jacobian by application of $P$ gives the remaining rows in \eqref{eq:system_jac}.
\end{proof}

\subsection{Equivalence with implicit functions}\label{sec:equivalence}
In the previous section, we constructed a polynomial system from the NEPv \eqref{eq:full_NEPv_problem}, involving $\mu_1,\ldots,\mu_m,\lambda$ and $m$ equations.
This polynomial system does not contain the solution $v$.
Since we have $m+1$ variables and $m$ equations, the implicit function theorem can be applied whenever the Jacobian is non-singular.
This leads to implicit functions $\mu_1(\lambda),\ldots,\mu_m(\lambda)$, that define a NEP \eqref{eq:NEP}, whose solutions coincide with the solutions to \eqref{eq:full_NEPv_problem}, in a neighborhood of a given non-singular point.

\ejcomment{\begin{theorem}\label{thm:main_thm}
   Let $(\tilde{\lambda}, \tilde{\mu})$ be a solution to the polynomial system \eqref{eq:reduced_syst_mat_form} and assume the corresponding system Jacobian \eqref{eq:system_jac} is nonsingular.
\begin{itemize}
    \item [(a)] Then there exist unique functions $\mu(\lambda) = \trans{[\mu_1(\lambda), \dots, \mu_m(\lambda)]}$, continuous in a neighborhood of $\tilde{\lambda}$, denoted $D\subset\RR$, such that $(\lambda, \mu(\lambda))$ satisfies \eqref{eq:reduced_syst_mat_form} for all $\lambda\in D$, and $\tilde{\mu} = \mu(\tilde{\lambda})$. 
\end{itemize} 
   Moreover, for any $(\lambda_\ast, v_\ast)$ such that $\lambda_\ast\in D$, the following two statements are equivalent.
\begin{itemize}
    \item[(b)] The pair $(\lambda_\ast, v_\ast)$ is a solution to  \eqref{eq:full_NEPv_problem}.
    \item[(c)] The pair $(\lambda_\ast, v_\ast)$ is a solution to  \eqref{eq:NEP} with the functions defined by $\mu$ and $\norm{v_*}_B^2 = 1$.
\end{itemize}
\end{theorem}}
\ejcomment{\begin{proof}
    Statement (a) follows from the implicit function theorem \cite[Theorem~9.28]{Rudin:1976:PRINCIPLES}. More precisely,
    $\lambda$ is viewed as the free variable and the condition that the Jacobian with respect to the dependent variables $\mu_1,\ldots,\mu_m$ is nonsingular implies that functions having the stated properties exist.
    

    For the implications involving statement (b) and (c), we note that the NEP \eqref{eq:NEP} and the  NEPv 
    \eqref{eq:NEPv} 
    are the same if 
    \begin{equation}\label{eq:mu_is_aiv}
       \mu_i^2(\lambda_\ast) = (\trans{a_i}v_\ast)^2, \quad i=1,\dots,m.
    \end{equation}

   We now prove that (b) implies (c).  
   Since $(\lambda_*,v_*)$ are given as a solution to the NEPv \eqref{eq:NEPv}, we can define 
   $\mu_{*,i}:=\trans{a_i}v_*$ according to \eqref{eq:mui_def}.
   By Proposition~\ref{prop:polysys_satisfied}, $(\lambda_*,\mu_*)$ is a solution to 
   \eqref{eq:poly_sys_mat_form} and therefore also \eqref{eq:reduced_syst_mat_form}. Hence, by point (a) the unique functions satisfy \eqref{eq:mu_is_aiv},
   which implies that \eqref{eq:NEP} is satisfied.
   
   Finally we prove that (c) implies (b). 
    Assume that $(\lambda_\ast, v_\ast)$ is an eigenpair of \eqref{eq:NEP}, with $\lambda_\ast\in D$.
    Since $(\lambda_\ast, v_\ast)$ is an eigenpair of \eqref{eq:NEP}, we can rearrange the terms and obtain
    \begin{equation}
        v_\ast = (\lambda E-A_0)^{-1}\left(\mu_1^2(\lambda_\ast)(\trans{a_1}v_\ast)a_1 + \dots + \mu_m^2(\lambda_\ast)(\trans{a_m}v_\ast)a_m  \right).
    \end{equation}
    We can again define $\mu_{*,i}:=\trans{a_i}v_*$.
    Applying the normalization condition, together with multiplying from the left with $\trans{a_i}$, for $i=1, \dots, m$, and extracting $m$ rows using the matrix $P$, yields
    \begin{subequations}\label{eq:poly_sys_equiv_proof}
    \begin{eqnarray}
        & &\trans{[\mu(\lambda_\ast)^2\cdot\mu_*]}G(\lambda_\ast)[\mu(\lambda_\ast)^2\cdot\mu_*]-1 = 0,\\
        & &P\left(H(\lambda_\ast)[\mu(\lambda_\ast)^2\cdot\mu_*]-\mu_*\right) = 0.
    \end{eqnarray}
    \end{subequations}
    Since $\lambda_\ast\in D$, $\mu(\lambda_\ast)$ must satisfy both \eqref{eq:poly_sys_equiv_proof} and \eqref{eq:reduced_syst_mat_form} simultaneously.
    Hence, we have $\mu(\lambda_\ast) =\mu_*$, i.e., \eqref{eq:mu_is_aiv} is satisfied. 
\end{proof}}

\section{Solving the polynomial system}\label{sec:solving_poly_sys}

In order to develop methods based on the above construction, we now investigate and characterize the polynomial system.
When $m=2$, we have explicit solutions (given in Section~\ref{sec:two_nonlins}), and for the general case we propose an additional transformation (given in Section~\ref{sec:more_nonlins}) that 
can be used as a basis of method development. 

\subsection{Two nonlinear terms}\label{sec:two_nonlins}

\subsubsection{Explicit solution for two terms}\label{sec:two_nonlins_theory}
For $m=2$ the polynomial system can be approached analytically. 
The following corollary describes how the functions $\mu_1^2$ and $\mu_2^2$ are computationally tractable via the roots of a third-degree polynomial.

\begin{corollary}\label{cor:main_thm_cor}
    Suppose $D\subset\RR$, $\mu_1,\mu_2$ are as in Theorem~\ref{thm:main_thm}, with $m=2$, and $\lambda\in D$ is given.
    Denote by $h_{ij}$ and $g_{ij}$ the element $(i,j)$ of $H(\lambda)$ and $G(\lambda)$,
    respectively.
    Then $\mu_1^2(\lambda)=:\gamma$ satisfies the degree-three polynomial
    \begin{equation}\label{eq:corr_3_deg_pol}
        \gamma^3[h_{12}^2g_{11} - 2h_{12}h_{11}g_{12} + h_{11}^2g_{22}] + \gamma^2[2h_{12}g_{12}-2h_{11}g_{22}] + \gamma g_{22} - h_{12} = 0,
    \end{equation}
    and $\mu_2(\lambda)$ satisfies the equation 
    \begin{equation}\label{eq:mu_2_pol_two_nonlins}
        \mu_1-h_{11}\mu_1^3 = h_{12}\mu_2^3.
    \end{equation}
\end{corollary}
\begin{proof}
    The system of equations \eqref{eq:reduced_syst_mat_form} with this selection of $P$ leads to
    \begin{equation}
    g_{11}\mu_1^6+2g_{12}\mu_1^3\mu_2^3+g_{22}\mu_2^6=1,\label{eq:first_eq_deg_two_proof}
    \end{equation}
    and \eqref{eq:mu_2_pol_two_nonlins}.
    By inserting \eqref{eq:mu_2_pol_two_nonlins} into \eqref{eq:first_eq_deg_two_proof} after multiplication by $h_{12}^2$, we obtain
    \begin{equation}
        h_{12}^2g_{11}\mu_1^6+
        2h_{12}g_{12}\mu_1^3(\mu_1-h_{11}\mu_1^3)+
        g_{22}(\mu_1-h_{11}\mu_1^3)^2=h_{12}^2.
    \end{equation}
    Expanding the square, setting $\gamma = \mu_1^2$, and collecting terms yields \eqref{eq:corr_3_deg_pol}.
\end{proof}

\begin{figure}[H]
    \centering
    \includegraphics{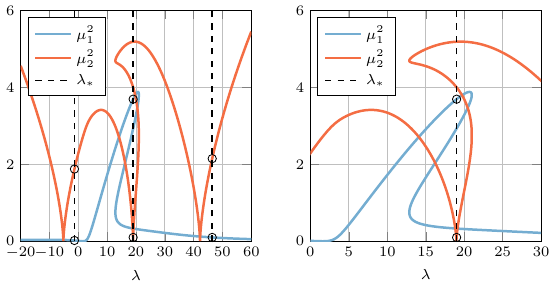}
    \caption{The curves $\mu_1^2$ and $\mu_2^2$ corresponding to the choice \eqref{eq:ill_ex_mats} plotted over an interval containing three eigenvalues.
    The black circles $(\boldsymbol{\circ})$ indicate which branch of the functions correspond to the computed eigenvalues.}
    \label{fig:mucurves_ill_ex}
\end{figure}

\subsubsection{Illustrative example}\label{sec:two_nonlins_ill_ex}
Let the matrix $A_0$ and the vectors $a_1,a_2$ be 
\begin{equation}
    A_0 = \begin{bmatrix}
        6 & 5 & 4 \\
        5 & 16 & 23\\
        4 & 23 & 20
    \end{bmatrix}, \: a_1 = \begin{bmatrix}
        2 \\ 0 \\ 0
    \end{bmatrix}, \: a_2 = \begin{bmatrix}
        0 \\ 2 \\ 0
    \end{bmatrix}.
    \label{eq:ill_ex_mats}
\end{equation}
We compute and plot the curves $\mu_1^2, \mu_2^2$ over an interval containing three eigenvalues of the NEPv corresponding to \eqref{eq:ill_ex_mats}. The curves are visualized in Figure~\ref{fig:mucurves_ill_ex}.
For this example, we discard the second equation in \eqref{eq:P_eqns}, and compute the eigenvalues by solving the equivalent NEP.
The three computed eigenvalues are given by $\lambda_{\ast,1} \approx -1.3447$, $\lambda_{\ast,2} \approx 19.0165$, $\lambda_{\ast,3} \approx 46.4337$, with the corresponding normalized eigenvectors 
$v_{\ast, 1} \approx \trans{[0.0708, -0.6851, 0.7250]}$, $v_{\ast,2} \approx \trans{[0.9611, -0.1575, -0.2269]}$, $v_{\ast,3} \approx \trans{[0.1577, 0.7330, 0.6617]}$.

Since $\mu_1^2$ and $\mu_2^2$ are algebraic functions, they possess branch point singularities, 
but are otherwise smooth functions of $\lambda$, making them suitable for use in general-purpose NEP solvers.
In these regions our approach defines a unique NEP.
However, there are also regions in which the functions become multi-valued. \ejcomment{This does not violate the theory concerning existence of the functions, since Theorem~\ref{thm:main_thm} only guarantees existence in a neighborhood of a nonsingular point in $(\mu_1,\ldots,\mu_m)$-space, and one fixed $\lambda$ can correspond to several such points. However, it does require  practical considerations in the context of an implementation. }
In these regions we can define more than one NEP, and a selection among them becomes necessary.
This can be handled in a variety of ways, for instance sorting them with respect to residual, or by proximity to some target value.
For this example, we have selected the branch closest to a target point. Although this multi-valued aspect of the approach is potentially a source of complication, it was easily manageable for the functions in Figure~\ref{fig:mucurves_ill_ex}.  We note that for the GPE-inspired example in Section~\ref{sec:num_examples} the functions appear to be single-valued.

\subsubsection{Perturbation analysis of singularities for two nonlinear terms}\label{sec:singular_two}
\ejcomment{The above construction has been theoretical in character. 
In practice we want to solve \eqref{eq:NEP} with a NEP-solver. 
The performance and reliability of NEP-solvers is affected by the presence and type of singularities. Branch point singularities become more problematic in this setting if the order is high, i.e., $(\cdot)^{1/k}$, where $k\gg 1$. In the following we show that the usual situation corresponds to $k=3$ for the constructed functions.}

In the illustrative example in Section~\ref{sec:two_nonlins_ill_ex}, we observe that singularities occur at points $\lambda_\ast$ where $\mu_2^2(\lambda_\ast)=0$.
In relation to Theorem~\ref{thm:main_thm}, this corresponds to a neighborhood where there exists no implicit function, due to the fact that when $\mu_2=0$, the Jacobian given by Proposition~\ref{thm:jacobian} has the structure
\[
J=\begin{bmatrix}\times & 0\\ \times & 0\end{bmatrix}.
\]
The Jacobian is singular and Theorem~\ref{thm:main_thm} is not applicable. 
It turns out that close to these singularities, $\mu_2^2$ behaves as 
$\mu_2^2=\mathcal{O}(|\lambda-\lambda_\ast|^{\frac{2}{3}})$.
We prove this asymptotic behavior, where we reuse the notation from Corollary~\ref{cor:main_thm_cor}.
First we will state a technical lemma (Lemma~\ref{lemma:mu1_analytic}) characterizing the admissible function values of $\mu_1^2$ and $\mu_2^2$ at such points. \ejcomment{This is subsequently used (in Proposition~\ref{thm:singpoints2_exp}) to determine the expansion, and the order of the branch point singularity.

To prove the following results, we need results from complex analysis. In particular we note that since the functions $\mu_1,\ldots,\mu_m$ are solutions to a polynomial system for every fixed $\lambda$, the functions can be extended analytically everywhere except for the branch point singularies and the poles.}

\begin{lemma}\label{lemma:mu1_analytic}
    Let $\mu_1, \mu_2$ be as in Theorem~\ref{thm:main_thm}. 
    Let $\lambda_\ast$ be such that $\mu^2_2(\lambda)\rightarrow 0$ as $\lambda\rightarrow\lambda_\ast$, and define $\mu_2^2(\lambda_\ast):=0$.
    Assume that 
    $h_{11}(\lambda_\ast)> 0$.
    If we define $\mu^2_1(\lambda_\ast) := 1/h_{11}(\lambda_\ast)$, then 
     $\mu_1$ is analytic in $\lambda_\ast$. 
\end{lemma}
\begin{proof}
    Taking $\lambda\rightarrow\lambda_\ast$ in \eqref{eq:mu_2_pol_two_nonlins} gives $\mu^2_1(\lambda_\ast) \rightarrow 1/h_{11}(\lambda_\ast)$, since we cannot have $\mu_1^2\rightarrow0$ and $\mu_2^2\rightarrow0$ simultaneously in \eqref{eq:first_eq_deg_two_proof}.
    Differentiating \eqref{eq:mu_2_pol_two_nonlins} with respect to $\lambda$, taking $\lambda\rightarrow\lambda_\ast$, and inserting the above values for $\mu_1^2(\lambda_\ast)$ and $\mu_2(\lambda_\ast)$, yields
    \begin{equation}
        \mu'_1(\lambda_\ast) = -\frac{h_{11}'(\lambda_\ast)}{h_{11}^{3/2}(\lambda_\ast)}.
    \end{equation}
    This implies that  $\mu_1$ is complex differentiable in $\lambda_\ast$, and it is therefore also analytic in $\lambda_\ast$, due to Riemann's removable singularity theorem \cite[Chapter 4, Theorem~7]{Ahlfors1978}.
\end{proof}
\begin{proposition}\label{thm:singpoints2_exp}
    Let $\mu_1^2(\lambda)$ and $\mu_2^2(\lambda)$ be as in 
    Lemma~\ref{lemma:mu1_analytic}.
    Assume that $h_{12}(\lambda_\ast)\neq0$.
    Then, we have
    \begin{equation}\label{eq:singpoints2_exp_formula}
    \mu_2^2(\lambda) = \mathcal{O}(|\lambda-\lambda_\ast|^{\frac{2}{3}}).
    \end{equation}
\end{proposition}
\begin{proof}
    Since $\mu_1$ is analytic we have a Taylor expansion and we can use the following expansions,
    \begin{eqnarray}
        \frac{\mu_1(\lambda)}{h_{12}(\lambda)} &=& \frac{\mu_1(\lambda_\ast)}{h_{12}(\lambda_\ast)} + K_1(\lambda-\lambda_\ast) + \mathcal{O}\left((\lambda-\lambda_\ast)^2\right),\\
        \frac{h_{11}(\lambda)}{h_{12}(\lambda)} &=& \frac{h_{11}(\lambda_\ast)}{h_{12}(\lambda_\ast)} + K_2(\lambda-\lambda_\ast) + \mathcal{O}\left((\lambda-\lambda_\ast)^2\right),
    \end{eqnarray}
    where $K_1,K_2$ are the expansion coefficients for the first-order term.
    Inserting these expansions into \eqref{eq:mu_2_pol_two_nonlins}, while neglecting higher-order terms, yields
    \begin{equation}
        \mu_2^3 = \frac{\mu_1(\lambda_\ast)}{h_{12}(\lambda_\ast)} + K_1(\lambda-\lambda_\ast) - 
        \left(\frac{h_{11}(\lambda_\ast)}{h_{12}(\lambda_\ast)} + K_2(\lambda-\lambda_\ast)\right)\left(\mu_1(\lambda_\ast) + \mu_1^{\prime}(\lambda_\ast)(\lambda-\lambda_\ast)\right)^3.
    \end{equation}
    Expanding the cube, distributing terms, setting $\mu^2_1(\lambda_\ast) = 1/h_{11}(\lambda_\ast)$, and again neglecting 
    higher-order terms, gives 
    \begin{equation}
        \mu_2^3 = (K_1-K_2\mu_1(\lambda_\ast))(\lambda-\lambda_\ast).
    \end{equation}
    From this we conclude \eqref{eq:singpoints2_exp_formula}.
\end{proof}

\subsection{More nonlinear terms ($m>2$)}\label{sec:more_nonlins}
For the case $m=2$, we showed in the previous section that the polynomial system could be solved analytically. 
We will show how the solutions to the polynomial system \eqref{eq:reduced_syst_mat_form} for $m>2$ can be obtained by means of a change of variables, and a reformulation as a multiparameter eigenvalue problem (MEP).

Consider the change of variables $w:=\mu^3$.
Then the solutions of the polynomial system \eqref{eq:reduced_syst_mat_form} are contained in those of a new system 
\begin{subequations}\label{eq:w_poly_system}
\begin{eqnarray}
    & &\trans{w}G(\lambda)w-1 = 0,\label{eq:first_w_eq}\\
    & &P\left((H(\lambda)w)^3-w\right) = 0\label{eq:second_w_eq},
\end{eqnarray}
\end{subequations}
obtained by cubing both sides of \eqref{eq:P_eqns}.
The solutions to \eqref{eq:w_poly_system} can in turn be extracted from the solutions of an associated MEP constructed in the following way.
We will use a companion linearization, a technique very common for polynomial eigenvalue problems, see, e.g., \cite{Mackey:2006:VECT}, but also polynomial multiparameter eigenvalue problems \cite{Muhic2010}.
Notice that \eqref{eq:first_w_eq} is equivalent to the system of equations
\begin{equation}
    \begin{bmatrix}
        -1 & \trans{w}\trans{G}\\
        w & -I_{m\times m}
    \end{bmatrix}\begin{bmatrix}
        1 \\ w
    \end{bmatrix} = 0,
    \label{eq:normalization_mep_eq}
\end{equation}
where $w$ appears only linearly in the coefficient matrix.
Furthermore, denoting the $k$th row of $H(\lambda)$ by $\trans{h_k}\in\mathbb{R}^{1\times m}$, the $k$th row of \eqref{eq:second_w_eq} is equivalent to the system of equations
\begin{equation}
    \begin{bmatrix}
        -w_k & 0 & \trans{h_k}w\\
        \trans{h_k}w & -1 & 0 \\
        0 & \trans{h_k}w & -1
    \end{bmatrix}\begin{bmatrix}
        1\\
        \trans{h_k}w \\
        (\trans{h_k}w)^2
    \end{bmatrix} = 0,
\label{eq:non_normalization_mep_eq}
\end{equation}
where $k=1,\dots,m-1$, assuming that $P$ discards the last equation in \eqref{eq:P_eqns}.
Together, \eqref{eq:normalization_mep_eq} and \eqref{eq:non_normalization_mep_eq} form $m$ equations, where the $m$ unknowns $w=[w_1, \dots, w_m]$ appear only linearly in the coefficient matrices.
This means that 
\eqref{eq:w_poly_system}
can be reformulated to the form of a general MEP, i.e., a problem of the form 
\begin{subequations}\label{eq:general_mep}
\begin{align}
    (A_{1,0} + w_1A_{1,1} +& \dots + w_mA_{1,m})x_0 = 0, \\
    (A_{2,0} + w_1A_{2,1} +& \dots + w_mA_{2,m})x_1 = 0, \\
    &\vdots \notag \\
    (A_{m,0} + w_1A_{m,1} +& \dots + w_mA_{m,m})x_{m-1} = 0,
\end{align}
\end{subequations}
where $x_0$ is the vector $\trans{[1, w]}$, $x_k=\trans{[1, \trans{h_k}w, (\trans{h_k}w)^2]}$, for $k=1,\dots,m-1$, and the matrices $A_{1,0}, \ldots, A_{m,m}$ are constructed correspondingly from \eqref{eq:normalization_mep_eq} and \eqref{eq:non_normalization_mep_eq}.
In this way, the solutions to the polynomial system can be obtained by solving the MEP \eqref{eq:general_mep}, and subsequently reversing the change of variables.

\begin{remark}[Singular points]\label{rem:singular}
In Section~\ref{sec:singular_two}, we showed, for $m=2$, that around points where $\mu_2=0$, corresponding to $\lambda=\lambda_*$, 
we have the asymptotic behavior 
$\mu_2^2(\lambda)=O(|\lambda-\lambda_*|^{2/3})$. A consistent conclusion can be drawn from the above construction using MEPs. 
Eigenvalue solutions to MEPs are generically analytic functions of the elements of the coefficient matrices $A_{1,0},\ldots,A_{m,m}$, which in turn are generically analytic functions of $\lambda$. \ejcomment{The generic situation occurs whenever the eigenvalues are simple, which is also the situation that we observed in the simulations.}
Therefore, near a point $w_i=0$, we have
\[
\mu_i^2=w_i^{2/3}=O(|\lambda -\lambda_*|^{2/3}).
\]
\ejcomment{Hence, the order of the branch point singularity is generically low, such that it can be treated with NEP-solvers.}
\end{remark}

\section{Implementation details}\label{sec:implementation}

In order to combine the techniques described in the previous sections into an effective strategy for solving \eqref{eq:full_NEPv_problem}, several implementation aspects need to be developed. In this section and the simulation section (Section~\ref{sec:num_examples}) we provide these details, which are also implemented and publicly available in an online repository\footnote{https://github.com/lithell/nepv-nep.git}. The simulations are implemented in Julia, specifically using \texttt{NEP-PACK} which is a  software package for NEPs \cite{Jarlebring:2018:NEPPACK}.


\subsection{Numerical methods for the polynomial system}\label{sec:num_methods_algebraic_system}

The scalar functions $\mu_1,\ldots,\mu_m$ are implicitly defined as
solutions to algebraic systems, e.g., in the general case \eqref{eq:reduced_syst_mat_form}, or as in Corollary~\ref{cor:main_thm_cor} for $m=2$. To apply the techniques in this article, these systems have to be solved efficiently and reliably. For $m=2$, we have an explicit form (Corollary~\ref{cor:main_thm_cor}), although this seemingly cannot be naturally generalized to $m>2$. 
Computationally solving algebraic systems of equations is the main topic of the field of computational algebraic geometry. 
The Julia package \texttt{HomotopyContinuation.jl} \cite{HomotopyContinuation.jl}
is one powerful tool that has emerged from this field.
We performed simulations with this tool, but an alternative approach based on the linearization, as described in Section~\ref{sec:more_nonlins}, turned out to be preferable in practice. 

In Section~\ref{sec:more_nonlins}, we illustrated how the algebraic system could be linearized (in the sense of companion linearization) to a MEP.
This opens up the possibility to use methods and software for MEPs, see e.g., \cite{Plestenjak:2016:NUMERICAL} and the references therein, as well as the Matlab software package \texttt{MultiParEig} \cite{multipareig}.
Since we needed functionality from the Julia ecosystem 
we could not directly use \texttt{MultiParEig}, but instead implemented a subset of the functionality for our setting. A fundamental technique for MEPs is the construction of
the generalized eigenvalue problems (GEPs)
\begin{equation}
      w_i \Delta_0 x  =  \Delta_i x \label{eq:delta_eq},\quad i=1,\dots,m,
\end{equation}
where $\Delta_0,\ldots,\Delta_m$ are matrices constructed from Kronecker products, derived from an operator determinant associated with the MEP \eqref{eq:general_mep}.
\ejcomment{In our context, we solve \eqref{eq:normalization_mep_eq} and \eqref{eq:non_normalization_mep_eq}, and therefore the $\Delta$-matrices are of dimension 
$(3^{m-1}(m+1)) \times (3^{m-1}(m+1))$, which unfortunately grows exponentially with the number of nonlinear terms, but is fortunately independent of the system size $n$.}
The eigenvalues of \eqref{eq:general_mep} can in turn be extracted from those of the GEPs \eqref{eq:delta_eq}.
Formally, \eqref{eq:delta_eq} is a family of generalized eigenvalue problems, with a shared eigenvector.
In practice, the consideration of only one of these generalized eigenvalue problems is often sufficient.
This yields one set of eigenvalues of the MEP, and the remaining eigenvalues can be computed, e.g., using Rayleigh quotients.
In our specific case, we chose to solve the GEP associated with the $\Delta_i$-matrix having the smallest condition number.
Again similar to \texttt{MultiParEig}, as a post-processing procedure, we carried out Newton refinement to improve the quality of the solutions.

\begin{remark}[Avoiding recomputation of $\mu_1,\ldots,\mu_m$]
Note that the evaluating $\mu_1,\ldots,\mu_m$ requires substantial computational effort, 
stemming from both the evaluation of the coefficients $G(\lambda)$ and $H(\lambda)$ as well as the solution to the MEP \eqref{eq:general_mep}. 
In the implementation in the context of a NEP-solver, e.g., as described in Section~\ref{sec:num_alg_for_neps}, we 
may need to repeatedly evaluate $\mu_1,\ldots,\mu_m$. In order to avoid recomputation, we use a look-up table to store already computed $\mu$-values. 
\end{remark}

\subsection{Adaption of iterative methods for NEPs}\label{sec:num_alg_for_neps}
The literature and results regarding numerical methods for NEPs of the general form 
\begin{equation}
    M(\lambda)v=0,
    \label{eq:general_nep_impldetails}
\end{equation}
where $M(\lambda)\in\CC^{n\times n}$, is extensive, see, e.g., \cite{Mehrmann:2004:NLEVP}, or the summary \cite{Guttel2017} for a more recent review of the field.
Consequently, there are many efficient methods for different use-cases.
We now motivate our choice of method, which we will subsequently discuss in more detail.

A large number of methods depend on the function $M(\lambda)$ being holomorphic or meromorphic in some domain $\Omega$ of the complex plane.
This class consists of, among other methods, contour-integral based methods \cite{Beyn:2011:INTEGRAL}\cite{Asakura:2009:NUMERICAL}, Krylov methods based on series expansion of $M(\lambda)$ (so-called infinite Arnoldi methods) \cite{Jarlebring:2012:INFARNOLDI}\cite{Jarlebring:2017:TIAR}\cite{Gaaf:2017:INFBILANCZOS}, and approaches based on rational approximation \cite{Guttel2017}\cite{Effenberger2012a}.
Crucially, for meromorphic problems, these methods depend on knowing the location of the singularities a priori.
In our application, this is not the case, and determining the locations of the singularities of $\mu^2_1(\lambda), \dots, \mu^2_m(\lambda)$ numerically is expensive.
Hence, we have not chosen to consider any of these methods for use in our application.

There are also a number of adaptions of Newton's method to nonlinear eigenvalue problems,  see, for instance, \cite{Guttel2017,Mehrmann:2004:NLEVP} or \cite{Ruhe:1973:NLEVP}. Although convergence of
Newton's method can only be guaranteed locally, Armijo steplength control (which is used in our setting) can be adapted to improve the convergence basin \cite{Kressner:2009:BLOCKNEWTON}.
This method class consists of locally convergent iterative schemes, whose main computational cost is the need to repeatedly solve an iteration-dependent linear system, involving the Jacobian matrix of the extended system 
\begin{equation}
    \begin{bmatrix} 
        M(\lambda)v\\
        \herm{c}v - 1
    \end{bmatrix} = 0,
    \label{eq:aug_system_newton_nep}
\end{equation}
for some vector $c\in\CC^n$ not orthogonal to the eigenvector $v$.
A common approach to reduce this computational cost is to consider so-called quasi-Newton methods.
In these methods, (parts of) the Jacobian matrix are kept constant throughout the iterations, meaning precomputed factorizations can be used to speed up the solution of the linear system.
In \cite{Jarlebring:2018:DISGUISED}, several choices of approximate Jacobians are considered.
Generally, the authors conclude that we can only expect local linear convergence to simple eigenvalues, where the convergence factor is proportional to the distance from the eigenvalue to some fixed shift parameter, $\sigma\in\CC$.
In our setting, the value in terms of computation time, of using a (partially) fixed Jacobian is less substantial due to the fact that every evaluation of $\mu_1^2, \dots, \mu_m^2$ requires the solution of $m$ (iteration-dependent) linear systems.
Instead, we have opted to use a locally quadratically convergent Newton scheme, specifically the augmented Newton method; see, e.g.,  \cite{Ruhe:1973:NLEVP}.
To allow the possibility of computing several eigenvalues, we will combine this with a deflation method to remove already computed eigenvalues from the solution set.
The deflation procedure will be described in Section~\ref{sec:deflation}.

Newton's method generates a sequence of eigenvalue and eigenvector approximations $(\lambda_k,v_k)$, $k=1,2,3,\dots$, and proceeds by solving the system 
\begin{equation}
    \begin{bmatrix}
        M(\lambda_k) & M'(\lambda_k)v_k \\
        \herm{c} & 0 
    \end{bmatrix} \begin{bmatrix}
        v_{k+1} - v_k \\
        \lambda_{k+1} - \lambda_k
    \end{bmatrix} = -\begin{bmatrix}
        M(\lambda_k)v_k \\
        \herm{c}v_k - 1
    \end{bmatrix}.
    \label{eq:augnewton_linsys}
\end{equation}
The augmented Newton method is obtained by equivalently writing these relations as the steps
\begin{eqnarray}
    & & M(\lambda_k)u_{k+1} = M'(\lambda_k)v_k,\label{eq:linsys_augnewton_final} \\
    & & \lambda_{k+1} = \lambda_k + \frac{1}{\herm{c}u_{k+1}},\\
    & & v_{k+1} = \frac{1}{\herm{c}u_{k+1}}u_{k+1}.
\end{eqnarray}
In this way, the augmented Newton method operates on vectors of length $n$, but is equivalent to Newton's method in exact arithmetic.

In the setting of our numerical simulations, we consider a discretization of a differential equation leading to $n\gg 1$, such that $A_0$ is a large and sparse matrix. Moreover, we assume that $m\ll n$, such that $M(\lambda_k)\in\RR^{n\times n}$ has a 
low-rank structure originating from the nonlinear terms in \eqref{eq:NEP}.
To avoid explicitly forming the rank-one matrices (which are dense matrices) involved in evaluating $M(\lambda_k)$ in \eqref{eq:linsys_augnewton_final}, we will employ the Sherman-Morrison-Woodbury (SMW) formula, see, e.g., \cite[Section 2.1.3]{Golub:2007:MATRIX}.
The SMW formula states that the inverse of the rank-$m$ update to $A_0$ required in solving \eqref{eq:linsys_augnewton_final}, can be computed by the identity
\begin{equation}
    (A_0 + U_k\trans{U_k})^{-1} = A_0^{-1}-A_0^{-1}U_k(I_{m\times m}+\trans{U_k}A_0^{-1}U_k)^{-1}\trans{U_k}A_0^{-1},
    \label{eq:smw}
\end{equation}
where we have defined $U_k=[\mu_1(\lambda_k)a_1, \dots, \mu_m(\lambda_k)a_m]\in\RR^{n\times m}$.
Notice that this means that the computation of $u_{k+1}$ in \eqref{eq:linsys_augnewton_final} can be performed in $1+m$ sparse $n\times n$ linear solves, instead of one dense linear solve.
In practice, using \eqref{eq:smw} is much more efficient for $m\ll n$, both with respect to computation time and memory limitations.
All reported results from the numerical simulations in Section~\ref{sec:num_examples} utilize \eqref{eq:smw} for the solution of linear systems involving low-rank terms. 

\subsection{Deflation}\label{sec:deflation}
In the numerical simulations in Section~\ref{sec:num_examples} we will employ a deflation technique developed in \cite{Effenberger2013}, based on the concept of invariant pairs for NEPs introduced in \cite{Kressner:2009:BLOCKNEWTON}. For simplicity, we consider the index-one case. 
The background necessary for constructing the deflation procedure is now briefly summarized, where we still consider the general setting of \eqref{eq:general_nep_impldetails}.
The deflation procedure defines an extended NEP one dimension larger.
The eigenvalues of the extended NEP are eigenvalues of the original NEP. The converse holds for all eigenvalues except for the deflated eigenvalue, which can be viewed as removed from the solution set.
For the definition of the extended problem, we will require the concept of minimal invariant pairs. 
Without loss of generality, we can assume that $M(\lambda)$ in \eqref{eq:general_nep_impldetails} has the structure $M(\lambda) = M_1f_1(\lambda) + \dots + M_mf_m(\lambda)$, where $M_1, \dots, M_m$ are matrices, and $f_1, \dots, f_m$ are scalar functions of the eigenvalue.
An invariant pair of \eqref{eq:general_nep_impldetails} is then defined to be a pair $(X,S)\in \RR^{n\times p}\times\RR^{p\times p}$ such that 
\begin{equation}
    M_1Xf_1(S) + \dots + M_mXf_m(S) = 0,
\end{equation}
where $f_i(S)$, $i=1,\dots,m$, are understood to be matrix functions.
In our case, $f_i = \mu_i^2$, $i=1,\dots,m$, which are generically analytic. Hence, the matrix functions $\mu_i^2(S)$ are well-defined.
The columns of $X$ form a basis for an invariant subspace of \eqref{eq:general_nep_impldetails}, and by means of a Schur decomposition of $S$, one can show that $S$ shares its eigenvalues with \eqref{eq:general_nep_impldetails}.
In \cite{Effenberger2013}, invariant pairs are used to construct a deflation technique, by expanding a previously computed invariant pair one column at a time.
We compute vectors $v,u$ and a scalar $\lambda$ such that the pair $(\widehat{X}, \widehat{S})$ defined by 
\begin{equation}
    (\widehat{X}, \widehat{S}) = \left(
        \begin{bmatrix}
            X & v
        \end{bmatrix}, 
        \begin{bmatrix}
            S & u \\
            0 & \lambda
        \end{bmatrix}
    \right),
    \label{eq:extended_minimal_pair}
\end{equation}
is also minimal, which is ensured by an orthogonality condition.
This orthogonality condition means we avoid reconvergence to an already computed eigenvalue.
If $(X,S)$ is a minimal invariant pair of \eqref{eq:general_nep_impldetails}, it is possible to show that all extended minimal pairs of the form \eqref{eq:extended_minimal_pair} are equivalent by similarity transform to such pairs with $u,v,\lambda$ defined as the solutions to the extended NEP
\begin{equation}
    \begin{bmatrix}
        M(\lambda) & U(\lambda) \\
        \herm{X}  & 0
    \end{bmatrix}\begin{bmatrix}
        v \\ u
    \end{bmatrix} = 0,
    \label{eq:extended_nep_deflation}
\end{equation}
with $\norm{v} + \norm{u} \neq 0$, and $U(\lambda)$ defined by
\begin{equation}
    U(\lambda) = M(\lambda)X(\lambda I -S)^{-1}.
\end{equation}
In this way, additional eigenvalues of the original problem \eqref{eq:general_nep_impldetails} can be extracted from those of \eqref{eq:extended_nep_deflation}, while avoiding reconvergence.
The foundational functionality of the above deflation technique is available in \texttt{NEP-PACK}.

\section{Numerical example}\label{sec:num_examples}
\subsection{Discretization of a PDE eigenvalue problem}
In order to illustrate the performance and properties
of the approach, we now study a 
GPE-inspired nonlinear eigenvalue problem stemming from the discretization of the following problem,
similar to the problem given in \cite[Section~2.3]{Henning:2025:REVIEW}.
Consider the two-dimensional nonlinear eigenvalue problem of finding $(u, \lambda)\in L^2(\Omega)\times\RR$ such that
\begin{equation}
    -\Delta u(x,y) + p(x,y)u(x,y) + \sum_{i=1}^m\phi_i^3(u)\psi_i(x,y)= \lambda u(x,y),
    \label{eq:cont_problem_numexp}
\end{equation}
with homogeneous Dirichlet boundary conditions and $m=5$.
Normalization of the solution is enforced with respect to the $L^2$-norm, i.e., we require $\norm{u}_{L^2} = 1$.
We consider the square domain $\Omega=[-1,1]\times[-1,1]$, and define the functions $\phi_i$ and $\psi_i$, $i=1,\dots,5$, by
\begin{equation}
    \phi_i(u)= \int_\Omega \psi_i(x,y)u(x,y)\,d\Omega, \quad \psi_i(x,y)=c_ie^{-\sigma_i [(x-x_i)^2 + (y-y_i)^2]},
\end{equation}
where we have selected the parameters to be $c_i=45$, $\sigma_i=6$, for $i=1,\dots,5$, and the centers of the Gaussians $\psi_i$ were selected as $(x_1,y_1) = (0.4,-0.6)$, $(x_2,y_2) = (0.6,0.3)$, $(x_3,y_3) = (0.1,0.6)$, $(x_4,y_4) = (-0.5,0.4)$, $(x_5,y_5) = (-0.4,-0.4)$.
The positions of the centers is essentially arbitrary, but were chosen in an approximately circular configuration in order to illustrate problem properties.
The potential function $p(x,y)$ is chosen to model a superposition of a harmonic trapping potential and an optical lattice, common in experiment \cite{Henning:2025:REVIEW}, although the parameters are chosen differently for illustration purposes.
More precisely, we take the potential to be $p(x,y) = p_\text{h}(x,y) + p_{\text{opt}}(x,y),$ where the harmonic oscillator potential $p_\text{h}$ is given by 
$p_\text{h}(x,y) := 16\left(\gamma_x^2x^2 + \gamma_y^2y^2\right)$,
with $\gamma_x=1, \gamma_y=2$, and the optical lattice potential $p_\text{opt}$ is 
$p_\text{opt}(x,y) := 64\left(\sin(4\pi x)^2 + \sin(4\pi y)^2\right)$.
With a uniform mesh, $\Delta x = \Delta y =: h$, and $(N+2)^2$ total points, we define, for the $N^2$ interior points, the discretized solution vector $v \approx [u(x_1, y_1), \dots, u(x_{N}, y_1),\dots, u(x_1, y_{N}), \dots, u(x_{N}, y_{N})]^T$ 
We set $L_{N^2}\in\RR^{N^2\times N^2}$ to be the standard two-dimensional discrete Laplacian stencil of size $N^2$, i.e.,
\[
    L_{N^2} = D_{2,N}\otimes I +  I\otimes D_{2,N},
\]
with $D_{2,N} = \frac{1}{h^2}\text{tridiag}(1, -2, 1)\in\RR^{N\times N}$, and we let $V$ be the diagonal matrix 
$V = \diag(p(x_1, y_1), \dots p(x_N,y_1), \dots, p(x_1, y_N), \dots, p(x_N,y_N))$.
The integrals are approximated with the two-dimensional trapezoidal rule.
We arrive at the discrete problem of finding $(v, \lambda)\in\RR^{N^2}\times\RR$ such that
\begin{subequations}\label{eq:discrete_numexp}
\begin{eqnarray}
    & & \left(h^2(-L_{N^2} + V) + \sum_{i=1}^m(\trans{a_i}v)^2a_i\trans{a_i}\right)v = \lambda E v, \\
    & & \norm{v}_B^2 = 1,
\end{eqnarray}
\end{subequations}
with $a_i = h^2\trans{[\psi_i(x_1,y_1), \dots, \psi_i(x_N, y_1), \dots, \psi_i(x_1,y_N), \dots, \psi_i(x_N, y_N)]}$, and $E = B = h^2I$.

\begin{figure}[t]
    \centering
    \includegraphics{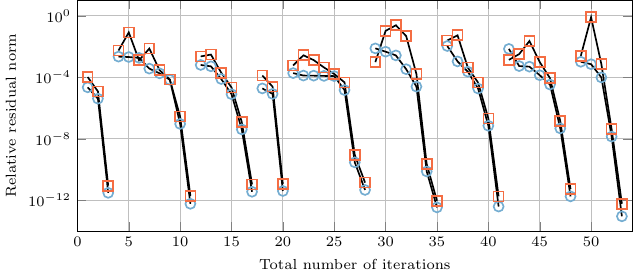}
    \caption{Convergence history for the problems \eqref{eq:discrete_numexp} and \eqref{eq:discrete_numexp_NEP}. 
    The relative residual norm of the NEP \eqref{eq:relres1} is plotted with blue circles $($\textcolor{myblue}{$\boldsymbol{\circ}$}$)$, and the corresponding relative residual norm of the NEPv \eqref{eq:relres2} is indicated by the red squares $($\textcolor{myred}{$\boldsymbol{\square}$}$)$.
    The stopping tolerance in the NEP solver was selected to be $5\cdot10^{-12}$ with respect to the relative residual norm.
    The eigenvalues were computed in the order indicated in Figure~\ref{fig:numexp}.
    }
    \label{fig:numexp_conv_hist}
\end{figure}

\subsection{Simulation results}
We apply our approach to the NEPv \eqref{eq:discrete_numexp}, leading to the NEP  given by
\begin{equation}\label{eq:discrete_numexp_NEP}
    \left(h^2(-L_{N^2} + V) + \sum_{i=1}^5\mu_i^2(\lambda)a_i\trans{a_i}\right)v = \lambda E v,
\end{equation}
where $\mu_1^2(\lambda), \dots, \mu_5^2(\lambda)$ are defined as in Section~\ref{sec:transformation}.
For the simulation we selected $N=256$, i.e., the problems \eqref{eq:discrete_numexp} and \eqref{eq:discrete_numexp_NEP} are of dimension $n=256^2 = 65536$.
The problem \eqref{eq:discrete_numexp_NEP} is solved by the augmented Newton method for nonlinear eigenvalue problems from Section~\ref{sec:num_alg_for_neps}.
To compute several eigenvalues of \eqref{eq:discrete_numexp_NEP}, we deflate already computed eigenvalues as described in Section~\ref{sec:deflation}.
The augmented Newton method requires the derivative of \eqref{eq:discrete_numexp_NEP} with respect to $\lambda$. 
Since the derivatives of $\mu_1^2, \dots, \mu_5^2$ are not directly available, these are approximated by standard second-order finite differences, which do introduce additional approximation error, but seemingly do not substantially influence the performance.
We selected the stopping tolerance in the NEP solver to be $5\cdot10^{-12}$, with respect to the relative residual norm of \eqref{eq:discrete_numexp_NEP}, i.e., the quality of the approximation of the pair $(\tilde{\lambda},\tilde{v})$ is quantified by
\begin{equation}\label{eq:relres1}
  \frac{\|M(\tilde{\lambda})\tilde{v} \|}{\|\tilde{v}\|}.
\end{equation}
For illustration we also studied the relative residual for the NEPv, i.e.,
\begin{equation}\label{eq:relres2}
    \frac{\|A(\tilde{v})\tilde{v}-\tilde{\lambda}\tilde{v} \|}{\|\tilde{v}\|}.
\end{equation}
All simulations were performed on a system with a 3.9GHz 8-core Intel i5-8265U processor and 8GB of memory, running Debian GNU/Linux 12.
The results of the simulation are presented in Figure~\ref{fig:numexp_conv_hist} and \ref{fig:numexp}.
Performance metrics from the simulation can be found in Table~\ref{table:performance_metrics}.
The functions $\mu_1^2,\dots,\mu_5^2$ for this example are visualized in Figure~\ref{fig:numexp_curves}. 

\begin{figure}[h]
    \centering
    \includegraphics{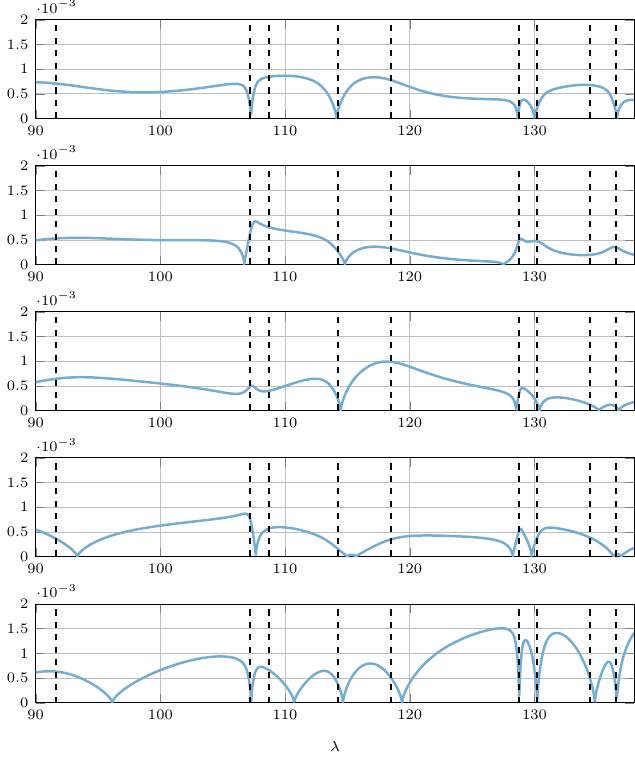}
    \caption{The curves $\mu_1^2,\ldots,\mu_5^2$, ordered top to bottom, for the example \eqref{eq:discrete_numexp_NEP}.
    For the relevant interval, the curves define single-valued functions, and correspondingly, a unique NEP.
    The computed eigenvalues are indicated by the vertical dashed lines.}
    \label{fig:numexp_curves}
\end{figure}

We observed rapid convergence to all nine computed eigenvalues once the solver entered the asymptotic regime.
As expected, the deflation procedure allowed us to robustly compute several eigenvalues of \eqref{eq:discrete_numexp}, even when the eigenvalues were not well separated.
In such a situation, a locally convergent NEPv solver applied directly to \eqref{eq:discrete_numexp} might struggle to distinguish between the eigenvalues unless several different starting guesses are used across multiple runs.
Clearly, instead solving the equivalent NEP, when combined with deflation, allows us to largely circumvent such difficulties.

Studying Figure~\ref{fig:numexp_curves}, we note that many eigenvalues appear close to singularities, which are also zeros of the functions $\mu_1,\ldots,\mu_5$ due to Remark~\ref{rem:singular}.
This behavior is expected and can be explained by studying Figure~\ref{fig:numexp}. 
For the eigenmodes corresponding to these eigenvalues, the Gaussians $\psi_i(x,y)$, corresponding to $\mu_i^2(\lambda)$, are nearly centered on the zero-valued interface between positive and negative parts of the eigenfunctions.
Hence, the $L^2$ inner product between $\psi_i(x,y)$ and the eigenmode is small, by the symmetry of $\psi_i(x,y)$.
Consequently, by the definition of $\mu_1^2,\dots,\mu_5^2$, we also expect these functions to be small for these eigenvalues.
The fact that we robustly found these eigenvalues, despite being close to singularities, can be seen as a favorable aspect of the method.

Another observation from Figure~\ref{fig:numexp_curves} is that this problem seems to have unique $\mu$-values. That is, although the polynomial system that is solved for every $\lambda$ has many complex solutions, there seems to only be one real-valued solution. This means that the potential complication with multi-valued $\mu$-functions is not a concern for this specific application.

Iteration times and a quantification of the number of linear solves are given in Table~\ref{table:performance_metrics}. We observe that the deflation technique is successful in this situation, since we can compute 9 eigenvalues, and the number of iterations per eigenvalue is not substantially increasing.  The CPU time per eigenvalue does increase  slightly. This can be partially explained by having to solve additional systems when using the extended NEP \eqref{eq:extended_nep_deflation}. However, the variation in CPU time is more influenced by the number of Armijo steps required, e.g., 
for eigenvalue $\lambda_2$ and eigenvalue $\lambda_5$ that have the same number of iterations, but differ significantly in terms of the number of linear solves. 

\vpcomment{In order to further illustrate the properties of our proposed method, we also performed simulations with the inverse iteration method ($J$-version) from \cite{Jarlebring:2014:INVIT}, where the aim was to compute the smallest eigenvalue of \eqref{eq:discrete_numexp}.
The $J$-method requires a shift parameter $\sigma$, and we chose this parameter in two different ways, $\sigma=50$ and $\sigma=90$, to 
illustrate two different representative situations. 
We again employed the SMW-formula for the solution of the linear systems, similar to \eqref{eq:smw}.
For the choice $\sigma=50$, we observed convergence to the same eigenvalue, $\lambda\approx91.6324$, as reported for our method. This required approximately 53 iterations and $318$ linear solves. We achieved a relative residual norm of approximately $8.3176\cdot 10^{-11}$.
For $\sigma=90$, we again observed convergence to the same eigenvalue as previously, after 33 iterations and $198$ linear solves, with a corresponding relative residual norm of approximately $6.7638\cdot10^{-11}$.
Clearly, we required a similar number of linear solves, as compared to our method to reach a comparable residual norm, cf. Table~\ref{table:performance_metrics}.
}

\begin{figure}[h]
    \centering
    \includegraphics[width=\textwidth]{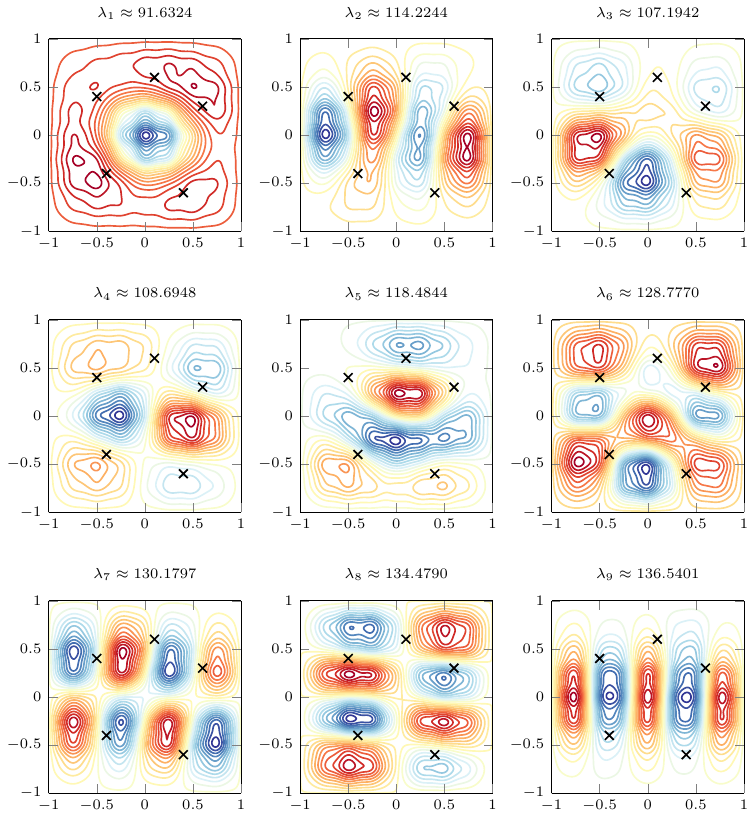}
    \caption{Normalized eigenmodes of the discrete problem \eqref{eq:discrete_numexp}, for $N=256$.
    The black crosses $($\ding{53}$)$ indicate the centers of the Gaussians $\psi_i(x,y)$, $i=1,\dots,5$.
    When read left to right, top to bottom, the eigenmodes are presented in the order in which the eigenvalues were computed using deflation. 
    The eigenvalues have been truncated to four decimal places.}
    \label{fig:numexp}
\end{figure}
\scriptsize
\begin{table}[H]
\centering
\scriptsize
{\renewcommand{\arraystretch}{1.2}
\begin{tabular}{c||c|c|c|c|c|c|c|c|c}
{\renewcommand{\arraystretch}{1.8}
\textbf{\begin{tabular}[c]{@{}c@{}} Eigenvalue \end{tabular}}} & $\boldsymbol{\lambda_1}$ & $\boldsymbol{\lambda_2}$ & $\boldsymbol{\lambda_3}$ & $\boldsymbol{\lambda_4}$ & $\boldsymbol{\lambda_5}$ & $\boldsymbol{\lambda_6}$ & $\boldsymbol{\lambda_7}$ & $\boldsymbol{\lambda_8}$ & $\boldsymbol{\lambda_9}$ \\ \hline
\textbf{\begin{tabular}[c]{@{}c@{}}No. of \\ iterations \end{tabular}} & 2 & 7 & 5 & 2 & 7 & 6 & 5 & 6 & 4 \\ \cline{1-1}
\textbf{\begin{tabular}[c]{@{}c@{}}CPU \\ time $\boldsymbol{[s]}$\end{tabular}} & 10.12 & 47.49 & 30.74 & 13.53 & 65.01 & 48.06 & 44.93 & 58.08 & 41.15 \\ \cline{1-1}
\textbf{\begin{tabular}[c]{@{}c@{}}No. of SMW \\ linsolves \end{tabular}} & 2 & 14 & 15 & 8 & 35 & 36 & 35 & 48 & 36 \\ \cline{1-1}
\textbf{\begin{tabular}[c]{@{}c@{}}No. of \\ $\boldsymbol{G(\lambda),H(\lambda)}$ evals. \end{tabular}} & 6 & 27 & 15 & 6 & 30 & 17 & 15 & 18 & 12 \\ \cline{1-1}
\textbf{\begin{tabular}[c]{@{}c@{}}Total no. of\\ linsolves \end{tabular}} & 42 & 219 & 165 & 78 & 360 & 301 & 285 & 378 & 276 \\ 
\end{tabular}
}
\vspace{0.5cm}
\caption{Performance metrics from the simulation. 
Each solve with the SMW formula involves six full-sized linear solves.
One evaluation of $G(\lambda)$ and $H(\lambda)$ can be performed together using five linear solves.
All linear solves are sparse.}
\label{table:performance_metrics}
\end{table}
\normalsize

\section{Conclusions and outlook}\label{sec:conc_and_out}

In this work, we have illustrated how a class of NEPvs can be transformed into a NEP in a way that enables the use of NEP solvers. The transformation involves a polynomial system of equations that must be solved each time the NEP is accessed. We propose several strategies for solving this system. For the case $m=2$, we derive an explicit solution, while for the general case $m>2$, we propose a transformation leading to a MEP that can be reliably solved with MEP methods. 

The approach seems generalizable in several ways. The fact that
we use scalar terms involving squares is due to the appearance of such nonlinearities in the motivating applications, and it can be naturally adapted to other polynomials, again leading to polynomial systems.
The strategy in this paper can also be adapted for more general choices of functions, although this will lead to a nonlinear system that is not necessarily polynomial.

A more general message of the results in this paper is the existence of a transformation from NEPv to NEP. In the proposed approach, evaluating the functions $\mu_1, \ldots, \mu_m$ requires substantial computational effort. Further research on the construction of similar functions may lead to alternative transformations with lower computational cost. For example, we note that $H(\lambda)$ is the transfer function of a linear dynamical system (and $G(\lambda)$ is related to a second-order system). Therefore, the evaluation could potentially be made more efficient by precomputing a reduced-order model \cite{Antoulas:2006:MOR,benner2005dimension}.

We note that the proposed approach is, to a large extent, independent of the choice of NEP solver. While we have suggested a specific solver in this paper, other options may be more appropriate for different problem types. For instance, if the singularities of the functions $\mu_1, \ldots, \mu_m$ can be determined without substantial computational effort, this may enable the use of rational approximation methods. See Section~\ref{sec:num_alg_for_neps} for further discussion.

The theoretical construction presented in this paper is applicable to an arbitrary number of terms, i.e., $m\gg 1$. However, when $m$ is large, the polynomial system becomes computationally restrictive to solve. Although other approaches than the one proposed here may be used to handle this, we expect the polynomial system to eventually become the dominating complication.
Note that the GPE corresponds to the case $m=n$ and is therefore not solvable with the current approach, at least not without further research.
The simulations in Section~\ref{sec:num_examples} suggest several such research directions.
The $\phi$-functions in \eqref{eq:cont_problem_numexp} can be interpreted as localized weighted averages of the wave function.
This suggests to iteratively update the $\psi$-functions in a fashion similar to the SCF iteration.
More knowledge about the approximation properties, in relation to the GPE, of an approach using a small number of nonlinear terms would be insightful in applying such an iterative technique.

\section*{Declarations}
\textbf{Conflicts of interest} The authors declare no conflict of interest.

\bibliographystyle{plain}
\bibliography{nepv,gpreview}

\begin{thebibliography}{10}

\bibitem{Ahlfors1978}
Lars~V. Ahlfors.
\newblock {\em Complex Analysis}.
\newblock McGraw-Hill, 3rd edition, 1978.

\bibitem{Antoulas:2006:MOR}
A.~Antoulas, D.~Sorensen, and S.~Gugercin.
\newblock A survey of model reduction methods for large-scale systems.
\newblock {\em Contemporary Mathematics}, 280:193--220, 2006.

\bibitem{Asakura:2009:NUMERICAL}
J.~Asakura, T.~Sakurai, H.~Tadano, T.~Ikegami, and K.~Kimura.
\newblock A numerical method for nonlinear eigenvalue problems using contour integrals.
\newblock {\em JSIAM Letters}, 1:52--55, 2009.

\bibitem{Bai:2022:SHARP}
Zhaojun Bai, Ren-Cang Li, and Ding Lu.
\newblock Sharp estimation of convergence rate for self-consistent field iteration to solve eigenvector-dependent nonlinear eigenvalue problems.
\newblock {\em SIAM Journal on Matrix Analysis and Applications}, 43(1):301--327, 2022.

\bibitem{Bai:2024:VARIATIONAL}
Zhaojun Bai and Ding Lu.
\newblock Variational characterization of monotone nonlinear eigenvector problems and geometry of self-consistent field iteration.
\newblock {\em SIAM Journal on Matrix Analysis and Applications}, 45(1):84--111, 2024.

\bibitem{Bai2018}
Zhaojun Bai, Ding Lu, and Bart Vandereycken.
\newblock Robust rayleigh quotient minimization and nonlinear eigenvalue problems.
\newblock {\em SIAM Journal on Scientific Computing}, 40(5):A3495--A3522, 2018.

\bibitem{Bao:2004:BOSEEINSTEIN}
W.~Bao and Q.~Du.
\newblock Computing the ground state solution of {Bose-Einstein} condensates by a normalized gradient flow.
\newblock {\em SIAM J. Sci. Comput.}, 25(5):1674--1697, 2004.

\bibitem{benner2005dimension}
Peter Benner, Volker Mehrmann, and Danny~C. Sorensen, editors.
\newblock {\em Dimension Reduction of Large-Scale Systems}.
\newblock Springer-Verlag, Berlin, Heidelberg, 2005.

\bibitem{Beyn:2011:INTEGRAL}
W.-J. Beyn.
\newblock An integral method for solving nonlinear eigenvalue problems.
\newblock {\em Linear Algebra Appl.}, 436(10):3839--3863, 2012.

\bibitem{HomotopyContinuation.jl}
Paul Breiding and Sascha Timme.
\newblock {H}omotopy{C}ontinuation.jl: {A} {P}ackage for {H}omotopy {C}ontinuation in {J}ulia.
\newblock In {\em International Congress on Mathematical Software}, pages 458--465. Springer, 2018.

\bibitem{Hein:2009:PLAPLACIAN}
T.~B\"{u}hler and M.~Hein.
\newblock Spectral clustering based on the graph {p-Laplacian}.
\newblock In {\em Proceedings of the 26th International Conference on Machine Learning}, pages 81--88, 2009.

\bibitem{Cai:2020}
Y.~Cai, Z.~Jia, and Z.-J Bai.
\newblock Perturbation analysis of an eigenvector-dependent nonlinear eigenvalue problem with applications.
\newblock {\em BIT}, 60:1--29, 2020.

\bibitem{Claes:2022:NEPvlin}
Rob Claes, Elias Jarlebring, Karl Meerbergen, and Parikshit Upadhyaya.
\newblock Linearizable eigenvector nonlinearities.
\newblock {\em SIAM Journal on Matrix Analysis and Applications}, 43(2):764--786, 2022.

\bibitem{Claes:2023:CONTOUR}
Rob Claes, Karl Meerbergen, and Simon Telen.
\newblock Contour integration for eigenvector nonlinearities.
\newblock {\em SIAM Journal on Matrix Analysis and Applications}, 44(4):1619--1644, 2023.

\bibitem{Jarlebring:2017:TIAR}
O.~Runborg E.~Jarlebring, G.~Mele.
\newblock The waveguide eigenvlaue problem and the tensor infinite {Arnoldi} method.
\newblock {\em SIAM J. Sci. Comput.}, 39:A1062--A1088, 2017.

\bibitem{Effenberger2013}
C.~{Effenberger}.
\newblock {Robust successive computation of eigenpairs for nonlinear eigenvalue problems}.
\newblock {\em SIAM J. Matrix Anal. Appl.}, 34(3):1231--1256, 2013.

\bibitem{Effenberger2012a}
C.~Effenberger and D.~Kressner.
\newblock Chebyshev interpolation for nonlinear eigenvalue problems.
\newblock {\em BIT}, 52(4):933--951, 2012.

\bibitem{Gaaf:2017:INFBILANCZOS}
S.~W. Gaaf and E.~Jarlebring.
\newblock {The infinite bi-Lanczos method for nonlinear eigenvalue problems}.
\newblock {\em SIAM J. Sci. Comput.}, 39(SIAM J. Sci. Comput.):S898--S919, 2017.

\bibitem{Golub:2007:MATRIX}
G.~Golub and C.~Van Loan.
\newblock {\em Matrix Computations}.
\newblock Johns Hopkins Univ. Press, 2007.

\bibitem{Guttel2017}
S.~Güttel and F.~Tisseur.
\newblock The nonlinear eigenvalue problem.
\newblock {\em Acta Numerica}, 26:1--94, 2017.

\bibitem{hein}
Matthias Hein and Thomas B\"uhler.
\newblock An inverse power method for nonlinear eigenproblems with applications in 1-spectral clustering and sparse pca.
\newblock In {\em Advances in Neural Information Processing Systems 23 (NIPS 2010)}, 2010.

\bibitem{Hen22}
Patrick Henning.
\newblock The dependency of spectral gaps on the convergence of the inverse iteration for a nonlinear eigenvector problem.
\newblock {\em Math. Models Methods Appl. Sci.}, 33(7):1517--1544, 2023.

\bibitem{Henning:2025:REVIEW}
Patrick Henning and Elias Jarlebring.
\newblock The {Gross-Pitaevskii} equation and eigenvector nonlinearities: Numerical methods and algorithms.
\newblock {\em SIAM Review}, 2025.

\bibitem{Henning:2020:SOBOLEV}
Patrick Henning and Daniel Peterseim.
\newblock Sobolev gradient flow for the gross--pitaevskii eigenvalue problem: Global convergence and computational efficiency.
\newblock {\em SIAM Journal on Numerical Analysis}, 58(3):1744--1772, 2020.

\bibitem{Jarlebring:2018:NEPPACK}
E.~Jarlebring, M.~Bennedich, G.~Mele, E.~Ringh, and P.~Upadhyaya.
\newblock {NEP-PACK}: A {Julia} package for nonlinear eigenproblems, 2018.
\newblock https://github.com/nep-pack.

\bibitem{Jarlebring:2018:DISGUISED}
E.~Jarlebring, A.~Koskela, and G.~Mele.
\newblock Disguised and new quasi-newton methods for nonlinear eigenvalue problems.
\newblock {\em Numer. Algor.}, 79(1):311--335, Sep 2018.

\bibitem{Jarlebring:2014:INVIT}
E.~Jarlebring, S.~Kvaal, and W.~Michiels.
\newblock An inverse iteration method for eigenvalue problems with eigenvector nonlinearities.
\newblock {\em SIAM J. Sci. Comput.}, 36(4):A1978--A2001, 2014.

\bibitem{Jarlebring:2012:INFARNOLDI}
E.~Jarlebring, W.~Michiels, and K.~Meerbergen.
\newblock A linear eigenvalue algorithm for the nonlinear eigenvalue problem.
\newblock {\em Numer. Math.}, 122(1):169--195, 2012.

\bibitem{Jarlebring:2021:IMPLICIT}
E.~Jarlebring and P.~Upadhyaya.
\newblock Implicit algorithms for eigenvector nonlinearities.
\newblock {\em Numer. Algor.}, 2021.

\bibitem{Kressner:2009:BLOCKNEWTON}
D.~Kressner.
\newblock A block {N}ewton method for nonlinear eigenvalue problems.
\newblock {\em Numer. Math.}, 114(2):355--372, 2009.

\bibitem{Lu:2024:ROSENBROCK}
D.~Lu, S.~Bora, A.~Prajapati, and P.~Sharma.
\newblock Eigenvalue backward error of rosenbrock systems and optimization of sums of rayleigh quotients.
\newblock Technical Report arXiv:2407.03784, arXiv, 2024.
\newblock arXiv preprint.

\bibitem{Lu2020}
Ding Lu.
\newblock Nonlinear eigenvector methods for convex minimization over the numerical range.
\newblock {\em SIAM Journal on Matrix Analysis and Applications}, 41(4):1771--1796, 2020.

\bibitem{Lu:2024:LOCALLY}
Ding Lu and Ren-Cang Li.
\newblock Locally unitarily invariantizable {NEPv} and convergence analysis of {SCF}.
\newblock {\em Mathematics of Computation}, 93(349):2291--2329, 2024.

\bibitem{Mackey:2006:VECT}
S.~Mackey, N.~Mackey, C.~Mehl, and V.~Mehrmann.
\newblock Vector spaces of linearizations for matrix polynomials.
\newblock {\em SIAM J. Matrix Anal. Appl.}, 28:971--1004, 2006.

\bibitem{Mehrmann:2004:NLEVP}
V.~Mehrmann and H.~Voss.
\newblock Nonlinear eigenvalue problems: A challange for modern eigenvalue methods.
\newblock {\em GAMM Mitteilungen}, 27:121--152, 2004.

\bibitem{Muhic2010}
Andrej Muhič and Bor Plestenjak.
\newblock On the quadratic two-parameter eigenvalue problem and its linearization.
\newblock {\em Linear Algebra and its Applications}, 432(10):2529--2542, 2010.

\bibitem{Pasadakis:2022:GRASSMANN}
Dimosthenis Pasadakis, Christie~Louis Alappat, Olaf Schenk, and Gerhard Wellein.
\newblock Multiway $p$-spectral graph cuts on {Grassmann} manifolds.
\newblock {\em Machine Learning}, 111(3):791--829, 2022.
\newblock Published online: 18 November 2021, Open Access.

\bibitem{Plestenjak:2016:NUMERICAL}
B.~Plestenjak.
\newblock Numerical methods for nonlinear two-parameter eigenvalue problems.
\newblock {\em BIT}, 56(1):241--262, 2016.

\bibitem{multipareig}
Bor Plestenjak.
\newblock Multipareig, 2025.
\newblock https://www.mathworks.com/matlabcentral/fileexchange/47844-multipareig.

\bibitem{Rudin:1976:PRINCIPLES}
W.~Rudin.
\newblock {\em Principles of mathematical analysis. 3rd ed}.
\newblock McGraw-Hill, 1976.

\bibitem{Ruhe:1973:NLEVP}
A.~Ruhe.
\newblock Algorithms for the nonlinear eigenvalue problem.
\newblock {\em SIAM J. Numer. Anal.}, 10:674--689, 1973.

\bibitem{Truhar:2021}
Ninoslav Truhar and Ren-Cang Li.
\newblock On an eigenvector-dependent nonlinear eigenvalue problem from the perspective of relative perturbation theory.
\newblock {\em Journal of Computational and Applied Mathematics}, 395:113596, 2021.

\bibitem{tudisco2018}
Francesco Tudisco, Francesca Arrigo, and Antoine Gautier.
\newblock Node and layer eigenvector centralities for multiplex networks.
\newblock {\em SIAM Journal on Applied Mathematics}, 78(2):853--876, 2018.

\bibitem{Upadhyaya:2021:DENSITY}
Parikshit Upadhyaya, Elias Jarlebring, and Emanuel~H. Rubensson.
\newblock A density matrix approach to the convergence of the self-consistent field iteration.
\newblock {\em Numerical Algebra, Control \& Optimization}, 11(1):99--115, 2021.

\bibitem{Yang:2009:SCF}
C.~Yang, W.~Gao, and J.~C. Meza.
\newblock On the convergence of the self-consistent field iteration for a class of nonlinear eigenvalue problems.
\newblock {\em SIAM J. Matrix Anal. Appl.}, 30(4):1773--1788, 2009.

\end{thebibliography}

\end{document}